\documentclass[reqno]{amsart}
\usepackage{amssymb}
\usepackage{url}

\numberwithin{equation}{section}
\newtheorem{thm}{Theorem}[section]
\newtheorem{prop}[thm]{Proposition}
\newtheorem{lem}[thm]{Lemma}
\newtheorem{cor}[thm]{Corollary}

\newtheorem{prob}[thm]{Problem}

\newcommand{\cref}[1]{Corollary~\ref{#1}}
\newcommand{\lref}[1]{Lemma~\ref{#1}}
\newcommand{\pref}[1]{Proposition~\ref{#1}}
\newcommand{\tref}[1]{Theorem~\ref{#1}}
\newcommand{\eref}[1]{(\ref{e#1})}
\newcommand{\secref}[1]{Section~\ref{#1}}

\newcommand{\bs}{\backslash}
\newcommand{\vhi}{\varphi}
\newcommand{\img}{\operatorname{Im}}
\newcommand{\mlt}{\operatorname{Mlt}}
\newcommand{\inn}{\operatorname{Inn}}
\newcommand{\atp}{\operatorname{Atp}}

\newcommand{\m}{^{-1}}

\newcommand{\aut}{\operatorname{Aut}}
\newcommand{\id}{\operatorname{id}}
\newcommand{\zbb}{\mathbb Z}
\newcommand{\eps}{\varepsilon}

\newcommand{\Sym}{\operatorname{Sym}}
\newcommand{\inv}{^{-1}}                % inverse
\newcommand{\ldv}{\backslash}           % left division
\newcommand{\rdv}{/}                    % right division

\newcommand{\midq}{;$ $} %{\mid}
\begin{document}
\title{Normality, nuclear squares and Osborn identities}
\author{Ale\v{s} Dr\'{a}pal}
\author{Michael Kinyon$^*$}
\thanks{${}^*$ Partially supported by Simons Foundation Collaboration Grant 359872}
\address{Dept.~of Mathematics \\ Charles University \\ Sokolovsk\'a 83 \\ 186
75 Praha 8, Czech Rep.}
\address{Dept.~of Mathematics \\ University of Denver\\ 2390 S.~York St. \\
Denver, CO 80208, USA}
\email{drapal@karlin.mff.cuni.cz}
\email{mkinyon@du.edu}

\begin{abstract}
Let $Q$ be a loop. If $S\le Q$ is such that $\vhi(S) \subseteq S$
for each standard generator of  $\inn Q$, then $S$ does not have to be
a  normal subloop. In an LC loop the left and middle nucleus coincide and
form a normal subloop. The identities of Osborn loops are obtained
by applying the idea of nuclear identification, and various connections
of Osborn loops to Moufang and CC loops are discussed. Every
Osborn loop possesses a normal nucleus, and this nucleus coincides
with the left, the right and the middle nucleus.
Loops that are both Buchsteiner and Osborn are characterized as loops
in which each square is in the nucleus.
\end{abstract}

\maketitle

Bruck's paper \cite{bruck} can be regarded as the beginning of the systematic theory
of loops. A relatively long development nonwithstanding, there still arise various lacunae
in the foundations. The intent of this paper is to fill some of these.
Standard references for loop theory are \cite{Belousov,bruck_book,pf} and essentially
all uncited claims can be found therein.

\section{Normality and inner mappings}\label{1}
Let $S$ be a subloop of a loop $Q$. It is well known \cite{pf}
that $S$ is normal if and only if $\vhi(S) = S$ for all $\vhi\in \inn Q$, where $\inn Q$
denotes the group of inner mappings. Since $\inn Q$ is closed for
inverses, the condition $\vhi(S) = S$ may be replaced by
$\vhi(S) \subseteq S$. The \emph{standard generators} of $\inn Q$
are $L_{xy}\m L_x L_y$, $R_{yx}\m R_x R_y$ and $L_x\m R_x$,
where $L_x\colon y \mapsto xy$ is the \emph{left translation} of the
element $x\in Q$, and $R_x\colon y \mapsto yx$ is the \emph{right translation}.
Recall that $\inn  Q$ is defined as $\{\vhi\in \mlt Q;$ $\vhi(1) = 1\}$,
where $\mlt Q = \langle L_x,R_x; x \in Q \rangle$.
The question addressed in this section is whether for $S$ to be normal
it suffices to assume that $\vhi(S) \subseteq S$ holds
for all standard generators $\vhi$.

It does not seem to be really surprising that the answer is negative.
Nevertheless, some effort seems to be needed to obtain an example that
is easy to describe. The example presented below was obtained
while investigating the structure of multiplicative equivalences \cite{dau}.

An equivalence $\sim$ upon the loop $Q$ is said to be
\emph{multiplicative} if
\begin{equation}\label{e11}
x\sim y \text{ and } u\sim v \quad \Longrightarrow \quad
xu\sim yv
\end{equation}
holds for all $x,y,u,v \in Q$. If $\sim$ is multiplicative
then $[x]_\sim\cdot [y]_\sim=[xy]_\sim$ is a well defined operation
upon $Q/\!\!\sim$ and
the set $S=[1]_\sim$ is closed under multiplication. However, $S$
is not necessarily closed under (left or right) division.

\begin{prop}\label{11}
Let $\sim$ be a multiplicative equivalence upon a loop $Q$. Put
$S=[1]_Q$  and assume that $xS=Sx=[x]_\sim$
for all $x\in Q$. Then $L_x\m R_x(S) = S$ for all $x\in Q$,
and if $\vhi$ is equal to $L_{xy}\m L_x L_y$ or $R_{yx}\m R_x R_y$,
then $\vhi(S)\subseteq S$, for all $x,y \in Q$.
Furthermore, $S$ is a subloop of $Q$.
\end{prop}
\begin{proof} First note that  $L_x\m R_x(S) = S$ is the same as $xS = Sx$.
Furthermore, $L_xL_y(S)=x[y]_\sim \subseteq [x]_\sim [y]_\sim
\subseteq [xy]_\sim = L_{xy}(S)$ for all $x,y\in Q$. To see that
$S$ has to be a subloop, consider $s,t \in S$. Since $sS = S = tS$,
there has to be $L_t\m L_s(S) = S$, and thus $t\bs s \in S$.
\end{proof}

\begin{cor}\label{12}
Let $\sim$ be a multiplicative equivalence upon a loop $Q$ that
is not a loop congruence. Put $S = [1]_\sim$, and suppose that
$xS = Sx = [1]_\sim$ for all $x\in Q$. Then there exist $x,y \in Q$
such that $L_{xy}\m L_x L_y(S)$ or $R_{yx}\m R_x R_y(S)$ is a proper
subset of $S$.
\end{cor}
\begin{proof} By \pref{11} the latter two sets have to be subsets of $S$.
If both of them are equal to $S$ for all $x,y\in Q$,
then $\vhi(S) = S$ for every standard generator $\vhi$ of $\inn Q$.
In such a case $S\unlhd Q$ and $\sim$ is a loop
congruence.
\end{proof}

Consider now a construction of loops
that generalizes an idea of Bates and Kiokemeister \cite{bk}.
The ingredients for the construction are a binary operation $\cdot$ upon
a set $A$,
a loop $Q$ and injective mappings $f_{a,b}\colon Q\to Q$. The underlying set
is equal to $A\times Q$.

The following statement coincides with
Theorem~4.1 of \cite{dau}.

\begin{prop}\label{13}
Let $\cdot$ be a binary
operation upon a set $A$ and let $1$ be the neutral element
of this operation. Suppose that mappings $x\mapsto ax$
and $x\mapsto xa$ are surjective upon $A$, for all
$a\in A$. Set
\begin{equation}\label{e12}
[c/ b] = \{a\in A;\ ab = c\} \text{ \ and \ }
[a \bs c]=\{b\in A;\ ab = c\},
\end{equation}
for all $a,b,c\in A$. For all $a,b\in A$ let
$f_{a,b}$ be an injective mapping $Q\to Q$, where $Q$
is a loop. Suppose that $f_{a,b}$ is the identity mapping
whenever $1\in \{a,b\}$. The element $(1,1)$ is then a neutral element of
the binary operation that is defined upon $A\times Q$ by
\begin{equation} \label{e13} (a,x)\cdot (b,y) = (ab,f_{a,b}(xy)).
\end{equation}
This operation yields a loop if and only if
\begin{enumerate}
\item[(1)] for all $c,b\in A$ the sets $\img(f_{a,b})$, $a\in[c/b]$,
are pairwise distinct and partition $Q$; and
\item[(2)] for all $a,c \in A$ the sets $\img(f_{a,b})$, $b\in [a\bs c]$,
are pairwise distinct and partition $Q$.
\end{enumerate}
\end{prop}

In the next two lemmas assume that $M$ is the loop upon $A\times Q$
that has been constructed by means of \pref{13}. Define an
equivalence $\sim$ upon $M$ by
\begin{equation}\label{e14}
(a,x)\sim (b,y) \ \Longleftrightarrow \ a=b.
\end{equation}

\begin{lem}\label{14}
If $(a,x) \in M$, then $[(a,x)]_\sim = \{a\}\times Q =
(a,x)(1\times Q) = (1\times Q)(a,x)$,
and $1\times Q = [(1,1)]_\sim$ is a subloop of $M$.
\end{lem}
\begin{proof} By \eref{13}, $(a,x)(1,y) = (a,f_{a,1}(xy)) = (a,xy)$.
Each element of $Q$ can be expressed as $xy$ for some $y\in Q$.
\end{proof}

\begin{lem}\label{15}
The equivalence $\sim$ is a multiplicative equivalence of $M$. It
is a congruence of $Q$ if and only if $A$ is a loop.
\end{lem}
\begin{proof} The projection $M\to A$, $(a,x)\mapsto
a$, is compatible with multiplication. Hence $\sim$ is
a multiplicative equivalence. This equivalence is a congruence
of $M$ if and only if $A$ is a loop since the multiplication
determines the divisions uniquely.
\end{proof}

\begin{thm}\label{16}
For a prime $p$ define an operation $\oplus = \oplus_p$ upon $\zbb$
by \begin{equation}\label{e15}
a\oplus b = \begin{cases} p\left (\left [ \frac a{p^2} \right ]
+ \left [\frac b{p^2}\right ]\right )
\text{ \ if $p\mid a{+}b$ and $p\nmid a$;}\\
a+b \text{ \  in every other case. }
\end{cases}
\end{equation}
Let an infinite loop $Q$ be partitioned to subsets $Q_i$ of the same
cardinality, $i\in \zbb_{p}$. Thus $Q = \bigcup Q_i$, and $Q_i\cap Q_j =
\emptyset$ if $i,j \in \zbb_{p}$ and $i\ne j$. For each
$i \in \zbb_{p}$ choose a bijection $\pi_i\colon Q\to Q_i$,
and set $\pi_a = \pi_i$ whenever $a\equiv i \bmod p$.

\noindent
Define an operation $\cdot$ upon $\zbb \times Q$ by
\begin{equation}\label{e16}
(a,x)(b,y) = \begin{cases}
(a\oplus b, \pi_{\frac{a+b-p}p}(xy)) \text{ \ if
$p\mid a{+}b$ and $p\nmid a$;}\\
(a\oplus b, xy) = (a+b,xy) \text{ \  in every other case. }
\end{cases}
\end{equation}
Then $M=(A \times Q,\cdot)$ is a loop with unit $(0,1)$ in which
$(a,x)\sim (b,y) \implies a =b$ defines a multiplicative
equivalence such that $S=[(0,1)]_\sim = \{0\}\times Q$ is subloop of $M$
that is not a normal subloop of $M$. If $\vhi$ is a standard generator
of $\inn M$, then $\vhi(S)\subseteq S$, and there exists a standard
generator $\vhi$ such that $\vhi(S) \ne S$.
If $Q$ is commutative, then $M$ is commutative too.
\end{thm}
\begin{proof}
Let us first verify that the construction of $M$ conforms with
\pref{13}. Consider $a,b\in \zbb$. If $p\nmid a{+}b$ or if
$p\mid a,b$ set $f_{a,b}=\id_Q$. In the other cases put
$f_{a,b}=\pi_{(a+b-p)/p}$.

Fix $a,c\in \zbb$ and consider, with respect to $\oplus$, the set $[a\bs c]$
as defined by \eref{12}.
If $p \nmid c$ and $a\oplus b = c$, then $c = a{+}b$  and
$[a\bs c] = \{c{-}a\}$. The same is true if
$p \mid a,c$. In these cases $f_{a,b} = \id_Q$. This means
that \eref{13} and \eref{16} yield the same result and that
condition (2) of \pref{13} gets satisfied.

Suppose now that $p\mid c$, $c=a\oplus b$ and $p\nmid a$.
Express $a$ as $a_2p^2 + a_1p + a_0$ and $b$ as $b_2p^2 + b_1p + b_0$,
where $\{a_0,b_0,a_1,b_1\} \subseteq \{0,\dots,p{-}1\}$. Then
$c = p(a_2{+}b_2)$ and $p=a_0+b_0$. The set $[a\bs c]$ is hence
equal to $\{cp-a_2p^2 +xp + p-a_0;$ $0\le x < p\}$. That can
be also expressed as $\{cp - a + p(a_1 {+} x {+} 1);$ $0\le x < p\}$.

The fact that $[a\bs c]$ is always nonempty means that
$u\mapsto a\oplus u$ is a surjective mapping $\zbb \to \zbb$ for
any $a \in \zbb$. Assume again that $p\mid c$, $p \nmid a$
and $b \in [a\bs c]$. The definition of $(a,x)(b,y)$ from
\eref{15} agrees with that of \eref{13}, by the choice of $f_{a,b}$.
What remains is to verify condition (2) of \pref{13}. If $b
= cp - a + p(a_1 {+} x {+} 1)$, then $(a+b-p)/p = c + a_1+x$. If
$x$ runs through $\zbb_p$, then $c+a_1+x$ runs through $\zbb_p$ as well.
Hence (2) holds.

Since mirror arguments are also true, $(M,\cdot)$ is a loop,
by \pref{13}. Lemmas~\ref{14} and~\ref{15} imply that \cref{12} can be used
to prove the rest.
\end{proof}

\section{Normal nuclei and LC loops}\label{2}

\emph{Left central} loops (or \emph{LC loops}) were introduced by Fenyves
\cite{f2}. There are several ways how they can be described \cite{f2,
pvv}. A unified treatment appears in
\tref{a2}. While \tref{a2} and \cref{a3} build
upon the existing concepts \cite{f2,pvv,pvc}, they contain several
characterizations that seem to be new.

In a loop $Q$, we will denote the left and right inverses of the element $x$
by $x^{\lambda} := 1\rdv x$ and $x^{\rho} := x\ldv 1$, respectively.

Let $Q$ be a loop. Then $N_{\lambda} = \{a\in Q;$ $a\cdot xy = ax\cdot y$
for all $x,y\in Q\}$ is known as the \emph{left nucleus} of $Q$.
Equations $x\cdot ay = xa\cdot y$ and $xy\cdot a = x\cdot ya$ yield
the \emph{middle} and the \emph{right} nucleus, respectively.
The main new result of this section is that in every LC loop the
left nucleus is a normal subloop (\tref{26}). The proof of this
fact starts with \pref{21} and does not depend upon the preceding
statements.

In each LC loop $N_{\lambda} = N_{\mu}$. This is because every LC loop $Q$
is a \emph{left inverse property} (LIP) loop, i.e.~it satisfies
$x^{\lambda}\cdot xy = y$ for all $x,y\in Q$. An equivalent characterization
of an LIP loop  is that for each $x\in Q$ there exists $y\in Q$ such
that $L_x\m = L_y$, where $L_x\colon a\mapsto xa$ is the \emph{left
translation} of the element $x\in Q$. Similarly, $Q$ is a RIP loop
if the inverse of each right translation $R_x\colon a\mapsto ax$
is also a right translation. Loops that are both LIP and RIP are
called \emph{inverse property} (IP) loops.

A loop $Q$ is said to be a \emph{left alternative property} (LAP) loop
if $x\cdot xy= xx\cdot y$ for all $x,y\in Q$. Clearly, $Q$ satisfies the LAP
if and only if $L_x^2$ is a left translation, for all $x\in Q$.
RAP loops are defined in a mirror way, and the intersection of
LAP and RAP loops is the variety of \emph{alternative property} (AP) loops.

Elements of nuclei can be characterized by means of autotopisms.
Let $Q$ be a loop. Denote by $S_Q$ the symmetric group upon $Q$.
A triple $(\alpha,\beta,\gamma)\in S_Q^3$ is an \emph{autotopism}
if $\alpha(x)\beta(y) = \gamma(xy)$ for all $x,y \in Q$.
Autotopisms of $Q$ form a group that is denoted by $\atp(Q)$.
For $x\in Q$ put $\lambda_x = (L_x,\id_Q,L_x)$, $\rho_x =
(\id_Q,R_x,R_x)$ and $\mu_x = (R_x\m, L_x,\id_Q)$.
The following facts are well known and easy to prove:

\begin{lem}\label{a1}
Let $a$ be an element of a loop $Q$. Then
$\lambda_a \in \atp(Q)$ $\implies$ $a\in N_{\lambda}$,
$\rho_a \in \atp(Q)$ $\implies$ $a\in N_{\rho}$ and
$\mu_a \in \atp(Q)$ $\implies$ $a\in N_{\mu}$.
Furthermore, consider $\sigma=(\alpha,\beta,\gamma)\in \atp(Q)$.
\begin{enumerate}
\item[(i)] $\alpha = \id_Q$ $\implies$ $\sigma = \rho_a$
for some $a\in N_{\rho}$;
\item[(ii)] $\beta = \id_Q$ $\implies$ $\sigma = \lambda_a$
for some $a\in N_{\lambda}$; and
\item[(iii)] $\gamma = \id_Q$ $\implies$ $\sigma = \mu_a$
for some $a\in N_{\mu}$.
\end{enumerate}
\end{lem}

Let $Q$ be an LIP loop. Then $x^{\lambda} = x^{\rho}$ for each $x\in Q$. Put
$I(x) = x\m = 1/x$. Then $I$ permutes $Q$ and $(\alpha,\beta,\gamma)
\in \atp(Q)$ $\iff$ $(I\alpha I,\gamma,\beta)\in \atp(Q)$.
\lref{a1} can be used to see that $a\in N_{\lambda}$ $\iff$
$(IL_aI,L_a,\id_Q)$ $\iff$ $a\in N_{\mu}$. Thus $N_{\lambda}
= N_{\mu}$ in every LIP loop $Q$. Similarly $N_{\rho} = N_{\mu}$ in an RIP loop.
Each IP loop thus contains an \emph{nucleus} $N_{\lambda} = N_{\mu} = N_{\rho}$.
An element belonging to the nucleus is said to be \emph{nuclear}.
Note that in a general setting the nucleus is defined as
$N_{\lambda} \cap N_{\mu} \cap N_{\rho}$.

\begin{thm}\label{a2}
Let $Q$ be a loop. The following conditions are equivalent:
\begin{enumerate}
\item[(1)] $x(x\cdot yz) = (x\cdot xy)z$ for all $x,y,z\in Q$;
\item[(2)] $xx\cdot yz = (x\cdot xy)z$ for all $x,y,z \in Q$;
\item[(3)] $(xx\cdot y)z = x (x\cdot yz)$ for all $x,y,z \in Q$;
\item[(4)] $y(x\cdot xz) = (y\cdot xx)z$ for all $x,yz,\in Q$;
\item[(5)] $Q$ is an LAP loop such that $x^2\in N_{\lambda}$ for each
$x\in Q$;
\item[(6)] $Q$ is an LAP loop such that $x^2\in N_{\mu}$ for each
$x\in Q$;
\item[(7)] $Q$ is an LIP loop such that $x^2\in N_{\lambda}$ for
each $x\in Q$;
\item[(8)] $\lambda_x^2 = (L_x^2,\id_Q,L_x^2)\in \atp(Q)$ for each
$x\in Q$;
\item[(9)] $L_x^2L_y$ is a left translation of $Q$, for all $x,y\in Q$; and
\item[(10)] $L_yL_x^2$ is a left translation of $Q$, for all $x,y \in Q$.
\end{enumerate}
\end{thm}
\begin{proof} First note that (8) is an equivalent expression of
 (1). The same is true for (9). Indeed, if $x(x\cdot y z) = wz$ for all
$z\in Q$, then the substitution $z=1$ implies that $w = x\cdot xy$.
Similarly, (4) $\iff$ (10).
Furthermore, note that setting $y=1$ or $z=1$ gives the LAP
for all of the identities (1--4). Using LAP it is clear that
(1) $\iff$ (2) and (1) $\iff$ (3). The
equivalence (2) $\iff$ (5) is immediate as well. Now, setting
$y=x^\rho =x\bs 1$ in (1) yields $x(x\cdot x^\rho z) = xz$. Thus
$x \cdot x^\rho z = z$ for all $x,z\in Q$. Hence (1) $\Rightarrow$ (7).
Since $N_{\lambda} = N_{\mu}$ in every LIP loop, the implication (7)
$\Rightarrow$ (5) follows from $x\cdot xy =x (x\m x^2 y) = x^2 y$.
We have shown the equivalence of (1), (2), (3), (5), (7), (8) and (9). The
next step is to prove  (4) $\iff$ (6). An equivalent
form of (4) is $(y/x^2)(x\cdot xz) = yz$. This can be expressed as
 $(R_{x^2}\m,L_x^2,\id_Q)\in \atp(Q)$,
for all $x\in Q$. In an LAP loop $L_x^2 = L_{x^2}$. Thus
(4) $\iff$ (6), by \lref{a1}. Furthermore, (5) $\Rightarrow$ (6)
since (5) $\iff$ (7), and $N_{\lambda} = N_{\mu}$ in every
LIP loop. To finish it thus suffices to prove that every loop
satisfying (6) is an LIP loop. Such a loop
fulfils $x^\lambda  x^2 = x$ since $x^\lambda  x^2  x^\rho
= x^\lambda \cdot x^2 x^\rho =x^\lambda(x\cdot x x^\rho) = 1 = x
x^\rho$.
By (4), $x^\lambda (x\cdot xz)=(x^\lambda x^2)z =
xz$ for all $x,z \in Q$, i.e., $Q$ satisfies the LIP.
\end{proof}

A loop is said to be an \emph{LC loop} if it satisfies the conditions
of \tref{a2}. The mirror conditions yield the \emph{RC loops}.
A loop is said to be a \emph{C loop} \cite{f2,pvc} if it is both
an LC loop and an RC loop. \tref{a2} easily yields the ensuing
characterization of C loops.

\begin{cor}\label{a3}
Let $Q$ be a loop. The following are equivalent:
\begin{enumerate}
\item[(1)] $Q$ is a C loop;
\item[(2)] $Q$ is an IP loop in which each square is nuclear;
\item[(3)] $Q$ is an AP loop with $x^2 \in N_{\mu}$ for all $x\in Q$;
\item[(4)] $(yx\cdot x)z = y(x\cdot xz)$ for all $x,y,z\in Q$;
\item[(5)] $\mu_x^2 = (R_x^{-2}, L_x^2,\id_Q) \in \atp(Q)$ for all $x\in Q$;
\end{enumerate}
\end{cor}
\begin{proof}
Equivalences (1) $\iff$ (2) and
(1) $\iff$ (3) follow from \tref{a2} in an immediate way.
The implication (3) $\implies$ (4) is also clear. Setting $y = 1$
and $z=1$ in (4) establishes the AP. Hence (3) $\implies$ (4).
Substituting $(y/x)/x$ for $y$ in (4) yields $yz = ((y/x)/x)\cdot
x(xz)$. That is the same as (5).
\end{proof}

The identity $xy \cdot zx = x(yz \cdot x)$ defines Moufang loops and
can be expressed by
saying that $\lambda_x \rho_x \in \atp(Q)$ for all $x\in Q$.
This observation served in \cite{nucid} as an impetus
to investigate all conditions
of the form $\sigma_x^\eps \tau_x^\eta \in\atp(Q)$
for each $x\in Q$, where $\{\eps,\eta\}\subseteq \{-1,1\}$, $\{\sigma,\tau
\}\subseteq \{\lambda,\rho,\mu\}$ and $\sigma\ne \tau$. It turns out
that the varieties obtained in this way are the varieties
of Moufang, left Bol, right Bol,
left conjugacy closed (LCC), right conjugacy closed
(RCC), Buchsteiner and
extra loops. In other words, these are the varieties can be obtained
by \emph{nuclear identification}.

The case when $\sigma = \tau$ and $\eps = \eta$ was not investigated
in \cite{nucid}. This may be regarded as an omission. Assume
$\eps = \eta = 1$ and $\sigma = \tau$. The case $\sigma = \lambda$
describes the LC loops, by \tref{a2}. The RC loops correspond
to the case $\sigma = \rho$, by a mirror argument. \cref{a3}
implies that the C loops can be obtained from the case $\sigma = \mu$.
The LC loops, RC loops and C loops thus result from a
nuclear identification as well. It is easy to verify that the
same varieties appear when $\sigma = \tau$ and $\eps=\eta=-1$
is assumed.

%The left Bol identity (LBol) can be expressed by saying that $L_xL_yL_x$
%is a left translation for all $x,y \in Q$. Similarly, $Q$ is LCC
%if $L_xL_yL_x\m$ is always a left translation. An alternative
%expression of LCC is that each $L_x\m L_y L_x$ is a left translation.
%Left Bol loops satisfy the LIP and the LAP. The left and the middle
%nuclei coincide in both LBol and LCC loops.

\begin{prop}\label{a4}
If a loop $Q$ satisfies two of the LCC, LC and LBol identities,
then it satisfies all three. A left Bol loop $Q$ is LC if and only
if $x^2\in N_{\lambda}$ for each $x\in Q$.
\end{prop}
\begin{proof} Every left Bol loop satisfies the LAP. Hence point (5)
of \tref{a2} can be used to prove that a left Bol loop $Q$ is an LC loop
if and only
if $x^2\in N_{\lambda}$ for each $x\in Q$.
The latter condition also characterizes
those left Bol loops that are LCC, e.g.~by formula (11) of \cite{nucid}.
A loop $Q$ that is both LCC and LC is an LCC loop that satisfies
the LAP and has
all squares in $N_{\lambda}$. Such a loop is left Bol, e.g.~by
formula (13) of \cite{nucid}.
\end{proof}

\begin{prop}\label{a5}
Let $Q$ be a loop. The following are equivalent.
\begin{enumerate}
\item[(i)] $Q$ is an extra loop;
\item[(ii)] $Q$ is an LC loop that is also a right Bol loop, or
an RCC loop, or a Buchsteiner loop; and
\item[(iii)] $Q$ is a C loop that is also a left Bol loop, or
an LCC loop.
\end{enumerate}
\end{prop}
\begin{proof} Extra loops are the Moufang loops with squares in the
nucleus \cite{f1}. Each extra loop is thus a C loop, by \cref{a3}.
Extra loops are conjugacy closed because a Moufang loop is
conjugacy closed if and only if all squares are in the nucleus, as
in \cite{f2} (cf.~formula (13) of \cite{nucid}). Thus (i) $\Rightarrow$ (iii).
The definition of extra loops is mirror symmetric. Thus also (i)
$\Rightarrow$ (iii'), where (iii') is the mirror version of (iii).
Denote by (ii') the condition (ii) from which the Buchsteiner clause
is removed. Trivially, (iii') $\Rightarrow$ (ii').
%Point (iii) follows from the mirror version of point (ii).
An LAP loop that is RCC or Buchsteiner is an extra loop,
by the mirror version of \cite[Corollary 2.5]{nucid}. Hence an LC
loop has to be an extra loop if it is a Buchsteiner loop or an RCC loop.
A right Bol loop with the LAP is Moufang \cite{pf}. Hence an
LC loop is extra if it is right Bol. Therefore (ii) $\Rightarrow$
(i). Both (i) $\iff$ (iii') and (i) $\iff$
(iii) follow. Thus if (iii) holds, then $Q$ is conjugacy closed. Conjugacy
closed loops with squares in the nucleus are Buchsteiner \cite[Theorem~3.3]{%
nucid}. Therefore (iii) $\Rightarrow$ (ii).
\end{proof}

Let $x$ be an element of loop $Q$. Then $T_x$ is defined as $R_x\m L_x$.
\begin{prop}\label{21}
Let $Q$ be a loop with a subloop $S$. Suppose that
$S\le N_{\lambda} \cap N_{\mu}$. If $T_x^{\pm 1}(s) \in S$ for
each $x\in Q$ and $s\in S$, then $S \unlhd Q$. Furthermore,
$L_{xy}\m L_xL_y(s) = T_{xy}\m T_x T_y(s)$ for all $x,y \in Q$ and $s\in S$.
\end{prop}
\begin{proof}
Suppose that $s \in S$ and $x,y \in Q$. Then
$sx\cdot y = s\cdot xy$ and so $R_{xy}\m R_y R_x(s) = s$.
Furthermore,
$$T_xT_y(s)\cdot xy = T_xT_y(s)x \cdot y = xT_y(s)\cdot y
= x \cdot T_y(s)y = x \cdot ys.$$
Therefore $T_{xy}\m T_x T_y(s) = (xy)\bs (T_xT_y(s)\cdot xy)
= xy \bs (x \cdot ys) = L_{xy}\m L_x L_y(s) \in S$.
\end{proof}

\begin{cor}\label{22}
Let $Q$ be a loop such that $N_{\mu} = N_{\lambda} \le Q$.
Then $N_{\mu}$ is a normal subloop if and only if
$T_x^{\pm 1}(a) \in N_{\mu}$ for all $x\in Q$ and $a \in N_{\mu}$.
\end{cor}

\begin{lem} \label{23}
Let $Q$ be a LIP loop such that $xax\m \in N_{\mu}$ for all $x\in Q$
and $a \in N_{\mu}$. Then $T_x\m(a) = x\m ax$
and $T_x(a) = xax\m$.
\end{lem}
\begin{proof}
We have $T_x\m(y) = x\bs (yx) = x\m \cdot yx $ for all $x,y \in Q$
since $Q$ is LIP. Put $y = xax\m$. By our assumptions, $y \in N_{\mu}$,
and so $T_x\m(xax\m) = x\m (xax\m \cdot x) = (x\m \cdot xax\m) x
= (x\m (x\cdot ax\m))x = (ax\m) x = a (x\m x) = a$. Therefore
$T_x(a) = xax\m$.
\end{proof}

\begin{cor}\label{24}
Let $Q$ be a LIP loop such that $xax\m \in N_{\mu}$ for all $x\in Q$
and $a \in N_{\mu}$. Then $N_{\mu} \unlhd Q$.
\end{cor}
\begin{proof} This follows directly from \lref{23}
and \cref{22} since $N_{\lambda} = N_{\mu}$ in every LIP loop $Q$.
\end{proof}

If $Q$ is an LIP loop, $x\in Q$ and $a \in N_{\mu}=N_{\lambda}$, then
$(ax)\m = x\m a\m$ as $(x\m a\m)(ax) = x\m x = 1$ and $(xa)\m = a\m x\m$
as $(xa)(a\m x\m) = 1$.

\begin{lem}\label{25}
Let $Q$ be an LC loop. Then $xax\m \in N_{\mu}$ for every
$x \in Q$ and $a \in N_{\mu}$.
\end{lem}
\begin{proof}
An LC loop is an LIP loop. Hence $N_{\lambda} = N_{\mu}$. All squares
of an LC loop belong to $N_{\lambda}$. Therefore $x ax\m \in N_{\mu}$ if
and only if $y^2 \cdot xax\m \in N_{\mu}$ for some $y \in Q$.
Put $y = (xa)\m = a\m x\m$. Then
$y^2 \cdot xax\m = y^2 \cdot y\m x\m = y\cdot (y\cdot y\m x\m)
= a\m x\m \cdot x\m = a\m x^{-2} \in N_{\mu}$.
\end{proof}

\begin{thm}\label{26}
If $Q$ is an LC loop, then $N_{\lambda} = N_{\mu} \unlhd Q$.
\end{thm}
\begin{proof} This is a straightforward consequence
of \lref{25} and \cref{24}.
\end{proof}

By \tref{26} the nucleus of a C loop is normal. This was first
proved by Phillips and Vojt\v echovsk\'y in \cite{pvc}.

\section{Osborn loops}\label{3}
A major motivation for studying Osborn loops is that they are a broad structured
variety of loops which include interesting classical varieties of loops, such as
Moufang loops and conjugacy closed loops, as special cases. Before turning to that,
we look at another approach. We informally mimic the scheme of nuclear identification
studied in \cite{nucid}.

Let $Q$ be a loop. For $x\in Q$,
let $\alpha_x, \beta_x\in \Sym{Q}$ satisfy $\alpha_x(1) = \beta_x(1) = x$.
If $(\alpha_x,\id_Q,\alpha_x)$ and $(\id_Q,\beta_x,\beta_x)$ are autotopisms, then it is
easy to see that $\alpha_x = L_x$ and $\beta_x = R_x$, in which case the autotopisms
are $\lambda_x$ and $\rho_x$, respectively. However, instead of assuming
at the outset that $\alpha_x$ and $\beta_x$ are translations, we leave them as permutations
to be determined. Analogous with the approach in \cite{nucid},
we assume that for each $x\in Q$,
\[
(\id_{Q},\beta_x,\beta_x)(\alpha_x,\id_{Q},\alpha_x) = (\alpha_x,\beta_x,\beta_x \alpha_x)
\]
is an autotopism, i.e.,
\[
\alpha_x (y)\cdot \beta_x (z) = \beta_x \alpha_x (yz)
\]
for all $x,y,z\in Q$. (The arbitrariness of the choice of the order in which we
multiplied the triples will be dealt with below.)

Setting $z=1$, we get $\alpha_x = R_x\inv \beta_x \alpha_x$, and so $\beta_x = R_x$.
Thus
\begin{equation}\label{Eq:osbtmp1}
\alpha_x(y) \cdot zx = \alpha_x(yz)\cdot x
\end{equation}
for all $x,y,z\in Q$. Taking $y=1$, we get $\alpha_x(z)\cdot x = x\cdot zx$.
This gives us $\alpha_x = R_x\inv L_x R_x$, but in the interest of easily finding
other expressions for $\alpha_x$, we will rewrite \eqref{Eq:osbtmp1} as
\begin{equation}\label{Eq:osbA}
  \alpha_x(y)\cdot zx = x(yz\cdot x)\,.
\end{equation}
Our automorphism is now
\[
\psi_x := ( \alpha_x, R_x, L_x R_x )\,.
\]

A loop $Q$ is said to be an \emph{Osborn loop} if for each $x\in Q$, there
exists $\alpha_x\in \Sym(Q)$ such that \eqref{Eq:osbA} holds. As noted,
$\alpha_x$ can be expressed in terms of translations so that \eqref{Eq:osbA}
is an identity in the language of loops and thus Osborn loops form a variety.
We next list a few different expressions for $\alpha_x$.

\begin{lem}\label{Lem:alpha}
  Let $Q$ be an Osborn loop and for each $x\in Q$, let $\alpha_x\in \Sym(Q)$
  be such that \eqref{Eq:osbA} holds. Then
  \[
  \alpha_x = R_x\inv L_x R_x = L_x R_x R_{x^{\lambda}} = L_{x^{\lambda}}\inv\,.
  \]
\end{lem}
\begin{proof}
  We have already noted the first equality, which can be obtained from
  \eqref{Eq:osbA} by taking $z = 1$. Instead taking $z = x^{\lambda}$, we have
  $\alpha_x = L_x R_x R_{x^{\lambda}}$. Now set $y = x^{\lambda}$ and note that
  $\alpha_x(x^{\lambda}) = R_x\inv L_x R_x(x^{\lambda}) = 1$. Thus $zx =
  x(x^{\lambda}z\cdot x)$, that is, $R_x = L_x R_x L_{x^{\lambda}}$. Rearranging this,
  we have $L_{x^{\lambda}}\inv = R_x\inv L_x R_x$, which completes the proof.
\end{proof}

It is worth recording separately a consequence of the last equality
of the preceding lemma.

\begin{cor}\label{Cor:RRLL}
  Let $Q$ be an Osborn loop. For all $x\in Q$,
  \begin{equation}\label{Eq:inverse_tmp}
  R_x R_{x^{\lambda}} L_{x^{\lambda}} L_x = \id_Q\,.
  \end{equation}
\end{cor}

Now for all $x$ in an Osborn loop $Q$, we have
$\psi_{x^{\rho}}\inv = (L_x, R_{x^{\rho}}\inv,R_{x^{\rho}}\inv L_{x^{\rho}}\inv)$.
From \eqref{Eq:inverse_tmp}, we obtain $R_{x^{\rho}}\inv L_{x^{\rho}}\inv = R_x L_x$,
and so
\begin{equation}\label{Eq:psiI}
\psi_{x^{\rho}}\inv = (L_x, R_{x^{\rho}}\inv,R_x L_x)
\end{equation}
is an autotopism.

Let $Q$ be a loop and let $(Q^{\mathrm{op}},\ast)$ denote the opposite loop defined by $x\ast y := yx$,
with translations $L^{\mathrm{op}}_x := R_x$ and $R^{\mathrm{op}}_x := L_x$ and inverses $x^{\hat{\lambda}} = x^{\rho}$
and $x^{\hat{\rho}} = x^{\lambda}$. A triple $(\alpha,\beta,\gamma)$ is an autotopism of $Q$ if and only if
$(\beta,\alpha,\gamma)$ is an autotopism of $Q^{\mathrm{op}}$. Thus $\psi_x\in \atp(Q)$ if and only if
$( L^{\mathrm{op}}_x, {R^{\mathrm{op}}_{x^{\hat{\rho}}}}\inv, R^{\mathrm{op}}_x L^{\mathrm{op}}_x)
\in \atp(Q^{\mathrm{op}})$, that is, if and only if $\psi_{x^{\hat{\rho}}}\in \atp(Q^{\mathrm{op}})$.
Since we already have that $\psi_x\in \atp(Q)$ if and only if $\psi_{x^{\rho}}\in \atp(Q)$, we
conclude that \emph{a loop} $Q$ \emph{is an Osborn loop if and only if its opposite loop} $Q^{\mathrm{op}}$
\emph{is an Osborn loop}. In particular, any Osborn identity is equivalent to its mirror, which is the
corresponding identity in $Q^{\mathrm{op}}$.

Informally, if we had multiplied the triples
$(\alpha_x,\id{Q},\alpha_x)$ and $(\id{Q},\beta_x,\beta_x)$ in the opposite order and assumed
that the product is an autotopism, we would have obtained $\alpha_x = L_x$, the autotopism
$( L_x, \beta_x, R_x L_x )$, and various expressions for the permutation
$\beta_x$. This is precisely \eqref{Eq:psiI}, and and so we would have
been led to the same variety of loops.

The various forms of the autotopisms $\psi_x$ and $\psi_{x^{\rho}}$ lead to
corresponding equivalent identities.

\begin{thm}\label{Thm:osb_equivs}
  The following identities are equivalent in loops.
  \begin{enumerate}
    \item $(x\cdot yx)\rdv x\cdot zx = x(yz\cdot x)$;
    \item $x(yx^{\lambda}\cdot x)\cdot zx = x(yz\cdot x)$;
    \item $x^{\lambda}\ldv y\cdot zx = x(yz\cdot x)$;
    \item $xy\cdot x\ldv (xz\cdot x) = (x\cdot yz)x$;
    \item $xy\cdot (x\cdot x^{\rho}z)x = (x\cdot yz)x$;
    \item $xy\cdot z\rdv x^{\rho} = (x\cdot yz)x$;
    \item $x^{\lambda}\ldv (x^{\lambda}y\cdot z) = (y\cdot zx)\rdv x$;
    \item $x \ldv (xy\cdot z) = (y\cdot zx^{\rho})\rdv x^{\rho}$.
  \end{enumerate}
\end{thm}
\begin{proof}
  Any two components of an autotopism uniquely determine the third. Identities (1), (2) and (3) are equivalent
  because they all express that $\psi_x$ is an autotopism with different forms of the first component $\alpha_x$.
  Identities (4), (5) and (6) are the mirrors of (1), (2) and (3), respectively.

  Next, starting with (3), replace $y$ with $x^{\lambda} y$ to get
  $yz\cdot x = L_x R_x(x^{\lambda}y\cdot z)
  = R_x L_{x^{\lambda}}\inv(x^{\lambda}y\cdot z)$, using Lemma \ref{Lem:alpha}.
  Applying $R_x\inv$ to both sides, we obtain (7). The steps are reversible, so
  (3) is equivalent to (7). Finally, (8) is the mirror version of (7).
\end{proof}

\noindent Any one of the identities in Theorem \ref{Thm:osb_equivs} may be taken as
the definition of Osborn loops. When we use these in what follows, we shall often just refer to ``an Osborn identity'' rather than singling out the particular form.

From comparing the identities of Theorem \ref{Thm:osb_equivs} with the Moufang
identity $xy\cdot zx = x(yz\cdot x)$ or its mirror image, we see that any Moufang
loop is an Osborn loop. We now give a more thorough characterization.

A loop $Q$ is \emph{flexible} (FLX) if $xy\cdot x = x\cdot yx$ for all $x,y\in Q$.

\begin{thm}\label{Thm:osb_mfg}
Any of the following are necessary and sufficient for an Osborn loop $Q$
to be a Moufang loop: (i) LIP, (ii) RIP, (iii) FLX,
(iv) LAP, (v) RAP.
\end{thm}
\begin{proof}
  Moufang loops are \emph{diassociative}, i.e., any $2$-generated subloop is
  associative. Conditions (i)--(v) are all particular instances of diassociativity
  so their necessity is clear.

  For sufficiency, (i), (ii) and (iii) are immediate from identities in Theorem \ref{Thm:osb_equivs}.
  For (iv): If LAP holds, then for all $x\in Q$,
  $x^{\lambda}\ldv x = L_{x^{\lambda}}\inv (x) = R_x\inv L_x R_x(x) =
  (x\cdot xx)\rdv x = (xx\cdot x)\rdv x = xx$. Thus for all $z\in Q$,
  $x(x\cdot zx) = xx\cdot zx = x^{\lambda}\ldv x\cdot zx = x(xz\cdot x)$ using
  an Osborn identity in the last equality. Cancelling we have $x\cdot zx
  = xz\cdot x$, that is, FLX holds. The proof of the sufficiency of (v)
  is dual to this.
\end{proof}

Other instances of diassociativity are also sufficient for an Osborn loop
to be a Moufang loop. For example, a loop $Q$ satisfies the
\emph{antiautomorphic inverse property} if $x^{\lambda} = x^{\rho}$ and
$(xy)^{\rho} = y^{\rho}x^{\rho}$ for all $x,y\in Q$. We omit the proof that an AAIP Osborn
loop is Moufang as it is a little more involved than the proofs of the five cases of
Theorem \ref{Thm:osb_mfg}.

As we will see, an instance of diassociativity which is not sufficient for
an Osborn loop $Q$ to be Moufang is the \emph{weak inverse property} (WIP):
$x (yx)^{\rho} = y^{\rho}$ or equivalently, $(xy)^{\lambda}x = y^{\lambda}$ for all $x,y\in Q$.

\medskip

The origin of Osborn loops lies in a paper of Osborn \cite{osborn},
who studied loops in which WIP holds in every loop isotope. He proved that
such a loop must satisfy the identity $xy\cdot \theta_x(z)x = (x\cdot yz)x$
where for each $x$, $\theta_x$ is a permutation.
Basarab \cite{basarab1} dubbed a loop satisfying the identity an Osborn
loop. In the same paper, Basarab also introduced generalized Moufang loops,
which we discuss further below. These turn out to be precisely WIP Osborn loops,
but not every Osborn loop has the WIP.

Independently and two years after Basarab's paper appeared, Huthnance also
studied what are now called Osborn loops in his PhD dissertation \cite{Huthnance}.
By an amusing coincidence, Huthnance reversed Basarab's terminology by using
``Osborn loops'' to refer to Basarab's generalized Moufang loops and
``generalized Moufang loops'' to refer to Basarab's Osborn loops. We follow
Basarab since his paper appeared first and other papers have since been published
following his convention \cite{basarab2,basarab3,basarab4,hruza}.
Many, but not all of the results in the remainder of this section can be found
in Basarab's papers.

\smallskip

If $(\alpha,\beta,\gamma)\in \atp(Q)$ is such that $\beta(1)=1$, then it is
well known and easy to show that $\alpha = \gamma = L_c \beta$ where $c = \alpha(1)$.
In this case, $\beta$ is called a \emph{left pseudoautomorphism} with \emph{companion} $c$.
A left pseudoautomorphism is an automorphism if and only if the companion lies in the
left nucleus.

Dually, if $(\alpha,\beta,\gamma)\in \atp(Q)$ is such that $\alpha(1) = 1$, then
$\beta = \gamma = R_c \alpha$ where $c = \beta(1)$. In this case, $\alpha$ is called
a \emph{right pseudoautomorphism} with \emph{companion} $c$.

A loop $Q$ is a $G$-\emph{loop} if it is isomorphic to all of its loop isotopes.
For instance, groups or more generally, conjugacy closed loops are $G$-loops.
It is well known that a loop is a $G$-loop if and only if each element occurs
as a companion of some left pseudoautomorphism and of some right pseudoautomorphism.

\begin{lem}
  Let $Q$ be a $G$-loop, i.e., for each $x\in Q$, there exists a left pseudoautomorphism
  $\varphi_x$ and a right pseudoautomorphism $\psi_x$, each with companion $x$.
  Assume further that $\varphi_x \psi_x = \psi_x \varphi_x = \id_Q$ and
  $\varphi_x(x) = x = \psi_x(x)$. Then $Q$ is an Osborn loop.
\end{lem}
\begin{proof}
Multiplying the autotopisms $( L_x \varphi_x, \varphi_x, L_x\varphi_x)$ and
$( \psi_x, R_x \psi_x, R_x \psi_x )$, we get that
$( L_x, \varphi_x L_x \psi_x, L_x \varphi_x R_x \psi_x )$ is an autotopism.
For each $y\in Q$, $L_x \varphi_x R_x \psi_x(y) =
L_x \varphi_x (\psi_x(y)\cdot x) = L_x \varphi_x \psi_x(y)\cdot \varphi_x(x)
= xy\cdot x$ since $\varphi$ is a left pseudoautomorphism and using the
assumptions of the lemma. Thus $L_x \varphi_x R_x \psi_x = R_x L_x$,
and so for each $x\in Q$, $( L_x, \varphi_x L_x \psi_x, R_x L_x )$ is
an autotopism. This autotopism has the form $(L_x, \beta_x, R_x L_x)$
for a permutation $\beta_x$ and so $Q$ is an Osborn loop.
\end{proof}

\begin{cor}\label{Cor:CC}
  Every conjugacy closed loop is an Osborn loop.
\end{cor}
\begin{proof}
  A loop $Q$ is conjugacy closed if and only if for each $x\in Q$,
  $T_x = R_x\inv L_x$ is a right pseudoautomorphism with companion $x$
  (this is LCC) and $T_x\inv$ is a left pseudoautomorphism with
  companion $x$ (this is RCC). CC-loops are well known to be $G$-loops.
  Since $T_x(x) = x = T_x\inv(x)$, the lemma applies.
\end{proof}

It is, in fact, easy to show from working directly with the autotopisms
defining LCC and RCC loops that an Osborn loop is CC if and only if
it is LCC if and only if it is RCC.

Basarab \cite{basarab5} defined a loop $Q$ to be a \emph{VD}-loop (probably named after
Valentin Danilovich Belousov) if every $T_x$ is a \emph{left} pseudoautomorphism
with companion $x$ and every $T_x\inv$ is a \emph{right} pseudoautomorphism with
companion $x$. Thus a VD-loop is defined by the identity
$x(xy\rdv x)\cdot (xz\rdv x) = x((x\cdot yz)\rdv x)$ and its mirror.
Every VD-loop is a $G$-loop, as follows from the characterization
of $G$-loops stated above. Moufang loops with nuclear fourth powers
and CC-loops with nuclear squares are VD-loops.

\begin{cor}[\cite{basarab5}]
  Every VD-loop is an Osborn loop.
\end{cor}

As discussed above, a \emph{generalized Moufang loop} is a WIP Osborn loop.
Such loops are characterized by the identity
$x(yz\cdot x) = (y^{\lambda}x^{\lambda})^{\rho}\cdot zx$ or its equivalent mirror. Indeed,
suppose this identity holds. Then for each $x$,
$( \rho R_{x^{\lambda}} \lambda, R_x, L_x R_x )$ is an autotopism. This has the form
$(\alpha_x, R_x, L_x R_x)$ for a permutation $\alpha_x$ and
so $\rho R_{x^{\lambda}} \lambda = L_x\inv$ by Lemma \ref{Lem:alpha}. This is precisely
WIP. The steps are clearly reversible, establishing the desired
characterization.

Osborn showed that for what we now call a generalized Moufang loop $Q$, the
factor loop $Q/N(Q)$ is a Moufang loop.
Basarab \cite{basarab1} proved that every isotope of a generalized Moufang
loop is a generalized Moufang loop. Every WIP CC-loop is a
generalized Moufang loop. In fact, combining results of Basarab  and
Goodaire and Robinson \cite{gr}, WIP CC-loops are precisely generalized
Moufang loops with every square in the nucleus. A short proof of the
latter fact appears in \secref{4}, cf.~\tref{43}. Generalized Moufang loops
satisfy a suitably generalized version of Moufang's theorem \cite{basarab_mfg}.

One can certainly obtain
generalized Moufang loops which are neither CC-loops nor Moufang loops
by, say, taking the direct product of a WIP CC-loop which is not an extra loop
and a Moufang loop which is not an extra loop. The smallest example given by this
construction is obtained from the nonassociative CC-loop of order $6$
and the nonassociative Moufang loop of order $12$. However it is not
known what is the order of the smallest generalized Moufang loop which is neither
a Moufang loop nor a CC-loop, nor has there ever been any effort at a classification of
generalized Moufang loops of small orders.

\medskip

Let $Q$ be a loop. Recall that the \emph{multiplication group} of $Q$ is
the permutation group $\mlt(Q) = \langle L_x,R_x \midq x \in Q \rangle$.
The \emph{left} and \emph{right multiplication groups} of $Q$ are, respectively,
$\mlt_{\lambda}(Q) = \langle L_x\midq x\in Q\rangle$ and
$\mlt_{\rho}(Q) = \langle R_x\midq x\in Q\rangle$.

\begin{thm}\label{Thm:normal}
  Let $Q$ be an Osborn loop. Then $\mlt_{\lambda}(Q)$ and $\mlt_{\rho}(Q)$ are
  normal subgroups of $\mlt(Q)$.
\end{thm}
\begin{proof}
  It is sufficient to show that for
  all $x,y\in Q$, $R_x\inv L_y R_x, R_x L_y R_x\inv\in \mlt_{\lambda}(Q)$ and
  $L_x\inv R_y L_x, L_x R_y L_x\inv\in \mlt_{\rho}(Q)$.
  The first follow immediately from writing Theorem \ref{Thm:osb_equivs}(7,8) in terms of translations:
  \begin{align}
  R_x\inv L_y R_x &= L\inv_{x^{\lambda}}L_{x^{\lambda}y}\,, \label{Eq:normal-tmp1} \\
  L_x\inv R_y L_x &= R\inv_{x^{\rho}} R_{yx^{\rho}}\,. \label{Eq:normal-tmp2}
  \end{align}
  Next rearrange \eqref{Eq:normal-tmp1} to
  get $R_x L_{x^{\lambda}y} R_x\inv = R_x L_{x^{\lambda}} R_x\inv L_y$. From Lemma
  \ref{Lem:alpha}, we have $R_x L_{x^{\lambda}} R_x\inv = L_x\inv$. Replacing $y$ with
  $x^{\lambda}\ldv y$, we get $R_x L_y R_x\inv = L_x\inv L_{x^{\lambda}\ldv y}$.
  Similarly, $L_x R_y L_x\inv = R_x\inv R_{y\rdv x^{\rho}}$, completing the proof.
\end{proof}

The \emph{left} and \emph{right inner mapping groups} $\inn_{\lambda}(Q)$ and
$\inn_{\rho}(Q)$
of a loop $Q$ are the stabilizers of $1$ in $\mlt_{\lambda}(Q)$ and $\mlt_{\rho}(Q)$,
respectively. In terms of generators, it turns out that
$\inn_{\lambda}(Q) = \langle L_{xy}\inv L_x L_y\midq x,y\in Q\rangle$ and
$\inn_{\rho}(Q) = \langle R_{yx}\inv R_x R_y\midq x,y\in Q\rangle$.

\begin{thm}\label{Thm:inner}
    Let $Q$ be an Osborn loop. Then for all $x,y\in Q$,
    \[
    [L_y, R_x] = (L_{x^{\lambda}y}\inv L_{x^{\lambda}} L_y)\inv = R\inv_{xy^{\rho}} R_{y^{\rho}}R_x\,.
    \]
    Therefore $\inn_{\lambda}(Q) = \langle [L_x, R_y] \midq x,y\in Q\rangle = \inn_{\rho}(Q)$.
\end{thm}
\begin{proof}
  The first equality follows from multiplying \eqref{Eq:normal-tmp1} on the left by $L_y\inv$.
  The second equality follow from exchanging $x$ and $y$ in \eqref{Eq:normal-tmp2}, multiplying
  on the left by $R_x\inv$ and then taking inverses of both sides.
  The remaining assertion follows
  from the characterizations of the left and right inner mapping groups in terms of generators.
\end{proof}

\begin{lem}\label{Lem:nucnorm}
  Let $Q$ be a loop. If $\mlt_{\lambda}(Q)\unlhd \mlt(Q)$, then the right nucleus of $Q$ is
  a normal subloop. If $\mlt_{\rho}(Q)\unlhd \mlt(Q)$, then the left nucleus of $Q$ is
  a normal subloop.
\end{lem}
\begin{proof}
  See (\cite{drapal-lcc}, Lemma 1.5).
\end{proof}

\begin{thm}\label{Thm:nucleus}
  Let $Q$ be an Osborn loop. Then the three nuclei of $Q$ coincide
  and the nucleus is a normal subloop.
\end{thm}
\begin{proof}
  The left nucleus of a loop $Q$ is the fixed point set of $\inn_{\lambda}(Q)$,
  the right nucleus is the fixed point set of $\inn_{\rho}(Q)$, and the middle
  nucleus is the fixed point set of $\langle [L_x, R_y]\midq x,y\in Q\rangle$.
  Thus the equality of the nuclei follows from Theorem \ref{Thm:inner}.
  The normality follows from Theorem \ref{Thm:normal} and Lemma \ref{Lem:nucnorm}.
\end{proof}

Call an Osborn loop \emph{proper} if it is neither conjugacy closed nor Moufang.
By exhaustive computer search, there are no proper Osborn loops up through order $13$.
The smallest known ones have order $16$,
and again by exhaustive computer search, they are the smallest which have
nontrivial nucleus. There are two of them, up to isomorphism. Each is a $G$-loop
and each contains the dihedral group $D_8$ as a subloop. Each has center
of order $2$, coinciding with the nucleus, and the factor by the center is
a nonassociative WIP CC-loop of order $8$. Each is nilpotent of class $3$;
the second center is a copy of $\mathbb{Z}_2$ and the factor by
the second center
is $\mathbb{Z}_4$. The two loops can be distinguished equationally; one
satisfies $L_x^4 = R_x^4 = \id$ but the other does not.

\begin{lem}\label{Lem:pseudo}
  Let $Q$ be an Osborn loop. Then for each $x,y\in Q$,
  \begin{enumerate}
    \item $L_{xy}\inv L_x L_y$ is a right pseudoautomorphism with companion $y\rdv x^{\rho}\cdot (xy)^{\rho}$;
    \item $R_{yx}\inv R_x R_y$ is a left pseudoautomorphism with companion $(yx)^{\lambda}\cdot x^{\lambda}\ldv y$.
  \end{enumerate}
\end{lem}
\begin{proof}
  We compute the autotopism
  \[
  \psi_{(xy)^{\rho}} \psi_{x^{\rho}}\inv \psi_{y^{\rho}}\inv =
    ( L_{xy}\inv L_x L_y, R_{(xy)^{\rho}} R_{x^{\rho}}\inv R_{y^{\rho}}\inv, \omega_{x,y} )\,,
  \]
  where we do not need the particular expression for $\omega_{x,y}$. The first
  component $L_{xy}\inv L_x L_y$ fixes $1$ and hence is a right pseudoautomorphism.
  The companion is $R_{(xy)^{\rho}} R_{x^{\rho}}\inv R_{y^{\rho}}\inv(1) = y\rdv x^{\rho}\cdot (xy)^{\rho}$,
  as claimed.
\end{proof}

\begin{cor}\label{Cor:aut}
  Let $Q$ be an Osborn loop. For each $x\in Q$, $L_{x^{\lambda}} L_x = L_x L_{x^{\rho}}$
  and $R_x R_{x^{\lambda}} = R_{x^{\rho}} R_x$ are automorphisms.
\end{cor}
\begin{proof}
  By Lemma \ref{Lem:pseudo}, $L_{x^{\lambda}} L_x$ is a left pseudoautomorphism with companion
  $x\rdv (x^{\lambda})^{\rho}\cdot (x^{\lambda}x)^{\rho} = 1$. Now $L_{x^{\lambda}}L_x L_{x^{\rho}}(y) = L_{x^{\lambda}}L_x(x^{\rho})\cdot L_{x^{\lambda}}L_x(y)
  = x^{\lambda}\cdot L_{x^{\lambda}}L_x(y) = L_{x^{\lambda}}L_{x^{\lambda}}L_x(y)$. Canceling $L_{x^{\lambda}}$, we obtain
  $L_{x^{\lambda}}L_x = L_x L_{x^{\rho}}$ as claimed. The remaining assertions follow dually.
\end{proof}

%In general, an Osborn loop does not have to satisfy any type of inverse property, not even
%WIP or its generalizations \cite{buch}. From Corollary \ref{Cor:aut} and Lemma \ref{Cor:RRLL},
%the only ``inverse-type'' property satisfied by all Osborn loops is
%\[
%L_{x^{\lambda}} L_x = (R_x R_{x^{\lambda}})\inv \in \aut(Q)\,.
%\]
%This is not unique to Osborn loops; for example, this also holds in all Buchsteiner loops \cite{buch}.

A loop $Q$ satisfies the \emph{crossed inverse property} (CIP) if $xy\cdot x^{\rho} = y$
or equivalently $x^{\lambda}\cdot yx = y$ for all $x,y\in Q$.

\begin{lem}\label{Lem:CIP}
  A CIP Osborn loop is a commutative Moufang loop.
\end{lem}
\begin{proof}
  For such a loop $Q$, from $y = xy\cdot x^{\rho}$, we have $x^{\rho}y = x^{\rho}(xy\cdot x^{\rho})
  = (x^{\rho})^{\lambda}\ldv x\cdot yx^{\rho} = 1\cdot yx^{\rho} = yx^{\rho}$, and so $Q$ is commutative.
  Commutative loops are flexible, so $Q$ is Moufang by Theorem \ref{Thm:osb_mfg}.
\end{proof}

A loop is \emph{left automorphic} or a \emph{left A-loop} if $\inn_{\lambda}(Q)\leq \aut(Q)$.
Left automorphic loops form a variety because the defining condition can be expressed equationally
using the generators: $L_{xy}\inv L_x L_y(zu) = L_{xy}\inv L_x L_y(z)\cdot L_{xy}\inv L_x L_y(u)$.
\emph{Right} automorphic loops are defined dually.

\begin{thm}\label{Thm:CML}
  Let $Q$ be a left or right automorphic Osborn loop. Then $Q/N(Q)$ is a commutative Moufang loop.
\end{thm}
\begin{proof}
  We prove the left case. We use Lemma \ref{Lem:pseudo} and the assumption that each
  $L_{x,y}$ is an automorphism to conclude that the companion $y\rdv x^{\rho}\cdot (xy)^{\rho}$
  lies in $N(Q)$.
  Thus in $Q/N(Q)$, the identity $y\rdv x^{\rho}\cdot (xy)^{\rho} = 1$ holds. This is equivalent to
  $xy = y\rdv x^{\rho}$ or $xy\cdot x^{\rho} = y$ for all $x,y\in Q$. By Lemma \ref{Lem:CIP}, we have the
  desired result.
\end{proof}

As a corollary, we obtain yet another proof of Basarab's CC-loop Theorem \cite{basarab-cc}.
Basarab's proof was explicated in a simplified form in \cite{kkp} and another proof was given in \cite{drapal-cc}.

\begin{cor}
  Let $Q$ be a CC-loop. Then $Q/N(Q)$ is an abelian group.
\end{cor}
\begin{proof}
  CC-loops are left automorphic, so by Corollary \ref{Cor:CC} and Theorem \ref{Thm:CML}, $Q/N(Q)$
  is commutative. Commutative CC-loops are abelian groups \cite{kunen}.
\end{proof}

Basarab \cite{basarab1} proved that if $Q$ is a loop such that every loop isotope is Osborn, then
$Q/N(Q)$ has WIP and hence is a generalized Moufang loop (and hence $(Q/N(Q))/N(Q/N(Q))$ is a
Moufang loops. However, he did not address the following, which is still the outstanding open
problem in the theory of Osborn loops.

\begin{prob}[\cite{kinyon-talk}]\label{Prb:isotopic}
   If $Q$ is an Osborn loop, is every loop isotopic to $Q$ an Osborn loop?
\end{prob}

\begin{prob}\label{Prb:simple}
  Does there exist a simple, proper Osborn loop?
\end{prob}

Note that an affirmative answer to Problem \ref{Prb:simple} would be a counterexample to
Problem \ref{Prb:isotopic}: if $Q$ is a simple, proper Osborn loop, then by Theorem \ref{Thm:nucleus},
$N(Q) = \{1\}$. If every loop isotopic to $Q$ is an Osborn loop, then by the discussion above,
$Q\cong Q/N(Q)$ is a generalized Moufang loop and hence a Moufang loop, contradicting the
assumption that $Q$ is proper.

\section{Loops that are both Osborn and Buchsteiner}\label{4}
As discussed in the last section, there are many equivalent ways to define Osborn loops.
We shall use the following characterization:
%\begin{equation} \label{e41}
%x((x^\lambda y)z\cdot x) = y \cdot zx.
%\end{equation}
%
%By setting $z=1$ we get $x\cdot (x^\lambda y)x = yx$, which
%can be written as $L_xR_x L_{x^\lambda} = R_x$. Thus
%\begin{equation}\label{e42}
%L_{x^\lambda}\m = R_x\m L_x R_x \text{\, if \,} x\in Q,\ Q
%\text{\, an Osborn loop.}
%\end{equation}
%
%Writing \eqref{e41} in the form $(x^\lambda y)(z/x) = (x\bs (yz))/x$
%we get the following characterization of Osborn loops:

\begin{equation}\label{e43}
Q \text{ is an Osborn loop} \ \iff \
\psi_x =(L\m_{x^\lambda},\, R_x,\, L_xR_x) \in \atp(Q)
\text{\, for all } x \in Q.
\end{equation}

Buchsteiner loops are characterized by the law
$x\bs(xy \cdot z) = (y \cdot zx)/x$. Setting $y= x$ yields
$L_x\m L_{x^2} = R_x\m L_x R_x$, while $z = x$ gives
$L_x\m R_x L_x = R_x\m R_{x^2}$. Thus
\begin{equation}\label{e44}
L_{x^2} = L_xR_x\m L_x R_x \text{\, and \,}
R_{x^2} = R_x L_x\m R_x L_x \text{\, if \,} x \in Q,
\ Q \text{\, a Buchsteiner loop.}
\end{equation}

A straightforward calculation using  \eqref{e44} yields
\begin{equation}\label{e45}
R_x^2 L_{x^2}\m L_x^2 = R_{x^2} \text{\, if \,} x \in Q,
\ Q \text{\, a Buchsteiner loop.}
\end{equation}

Both Osborn and Buchsteiner loops are known to have the
property that $N_{\lambda} = N_{\rho} = N_{\mu}$. In the following
we shall thus consider only the nucleus $N = N(Q)$.

\begin{equation}\label{e46}
L_{x^2} = L_xL\m_{x^\lambda} \text{\, if \,} x^2 \in N(Q),
\ Q \text{\, a loop.}
\end{equation}
This is true because (a) $x^2 \cdot x^\lambda = x$ since
$(x^2 \cdot x^\lambda) x = x^2 (x^\lambda \cdot x) = x^2$,
and so (b) $x^2 \cdot x^\lambda y = (x^2 \cdot x^\lambda)y = xy$.

\begin{equation}\label{e47}
L_{x^2} = L_xR_x\m L_x R_x \text{\, if \,} x^2 \in N(Q),
\ Q \text{\, an Osborn loop.}
\end{equation}
This follows from Lemma \ref{Lem:alpha} since $L\m_{x^\lambda} =
L_x\m L_{x^2}$, by \eqref{e46}.

The following fact can be obtained directly from the definition
of Buchsteiner loops and is well known:
\begin{equation}\label{e48}
Q \text{ is a Buchsteiner loop} \ \iff \
\varphi_x =(L_x,\, R_x\m,\, L_xR_x\m) \in \atp(Q)
\text{\, for all } x \in Q.
\end{equation}

\begin{thm}[Kinyon]\label{41}
Let $Q$ be a loop. Consider the following three properties:
\begin{enumerate}
\item[(a)] $Q$ is an Osborn loop;
\item[(b)] $Q$ is a Buchsteiner loop;
\item[(c)] $Q$ is a loop such that all squares are in the nucleus.
\end{enumerate}
Any two of these properties imply the third property.
\end{thm}
\begin{proof} Let $\varphi_x$ and $\psi_x$ be the triples
from \eqref{e43} and \eqref{e48}, $x\in Q$. Then
\[ \gamma_x = \varphi_x\psi_x = (L_xL\m_{x^\lambda},\,\id_Q,\,
L_xR_x\m L_x R_x).\]
If (a) and (b) are true, then $\gamma_x \in \atp(Q)$, and so
$L_xR_x\m L_xR_x(1) = x^2 \in N(Q)$. For the rest we can
assume that $x^2 \in N(Q)$ for all $x \in Q$. It will
suffice to show that $\gamma_x = (L_{x^2},\id_Q,L_{x^2})$
if $Q$ is an Osborn loop or a Buchsteiner loop.
First note that $x^2 \in N(Q)$ implies $L_xL\m_{x^\lambda} = L_{x^2}$,
by \eqref{e46}. If $Q$ is Buchsteiner, then
$L_{x^2} = L_xR_x\m L_x R_x$ by \eqref{e44}. The same equality
follows from \eqref{e47} if $Q$ is Osborn.
\end{proof}

\begin{thm}[Jaiyeola]\label{42}
Let $Q$ be a loop. Consider the following three properties:
\begin{enumerate}
\item[(a)] $Q$ is an Osborn loop;
\item[(b)] $Q$ is a Buchsteiner loop;
\item[(c)] $Q$ satisfies the law $(x\cdot xy)(x^{\lambda}\cdot xz)
= x(x\cdot yz)$.
\end{enumerate}
Any two of these properties imply the third property.
\end{thm}
\begin{proof}
First note that (c) is equivalent to the assumption that
\begin{equation}\label{e49}
\delta_x = (L_x^2,\, L_{x^\lambda}L_x,\,L_x^2) \in \atp(Q) \text{\,
for all \,} x\in Q.
\end{equation}
For every $x \in Q$
\[\varphi_x^{-2} \delta_x = (\id_Q,\,R_x^2 L_{x^\lambda}L_x,\,
R_xL_x\m R_x L_x).\]
If $Q$ is both Osborn and Buchsteiner, then $(\id_Q,R_{x^2},R_{x^2})
\in \atp(Q)$, by \tref{41}. In such a loop $\delta_x \in \atp(Q)$
if $\varphi_x^{-2}\delta_x = (\id_Q,R_{x^2},R_{x^2})$. Now,
$R_{x^2} = R_xL_x\m R_x L_x$ by \eqref{e44}, while \eqref{e46}
and \eqref{e45} imply $R_x^2 L_{x^\lambda}L_x= R_x^2
L_{x^2}\m L_x^2 = R_{x^2}$. Thus (a) and (b) imply (c).

If (a) and (c) hold, then $\varphi_x^{-2}\delta_x \in \atp(Q)$,
and that yields $x^2 = R_xL_x\m R_x L_x(1) \in N(Q)$. It remains
to consider the case when both (b) and (c) are true. For that
use that
\[\psi_{x^\rho}^2 \delta_x =
(\id_Q,\, R_{x^\rho}^2 L_{x^\lambda}L_x,\, (L_{x^\rho}R_{x^\rho})^2 L_x^2).\]
We obtain that  $R_{x^\rho}^2 L_{x^\lambda}L_x(1) = (x^\rho)^2 \in N(Q)$ for
all $x\in Q$.
\end{proof}

\begin{thm}\label{43}
Let $Q$ be a loop. Consider the following three properties:
\begin{enumerate}
\item[(a)] $Q$ is a generalized Moufang loop;
\item[(b)] $Q$ is a WIP CC loop;
\item[(c)] $Q$ is a loop such that all squares are in the nucleus.
\end{enumerate}
Any two of these properties imply the third property.
\end{thm}
\begin{proof}
This is a direct consequence of \tref{41} since
a WIP loop is a CC loop if and only if it is a Buchsteiner
loop, by \cite[Theorem 5.5]{nucid}, and since a loop is a generalized
Moufang loop if and only if it is a WIP Osborn loop.
\end{proof}

Note that \tref{43} remains true if point (c) is replaced by
the point (c) of \tref{42}. Note also that by \cite[Theorem 5.5]{nucid}
a WIP LCC loop is CC, and a WIP RCC loop is also CC.

\end{document}